\documentclass[english]{article}
\usepackage[T1]{fontenc}
\usepackage[latin9]{inputenc}
\usepackage{geometry}
\usepackage{float}
\usepackage{amsmath}
\usepackage{amssymb}
\usepackage{graphicx}

\makeatletter

\providecommand{\tabularnewline}{\\}

\newtheorem{theorem}{Theorem} 
\newtheorem{lemma}[theorem]{Lemma} 
 
\newtheorem{conjecture}[theorem]{Conjecture}
 
\newenvironment{proof}[1][Proof]
 {\begin{trivlist} \item[\hskip \labelsep {\bfseries #1}]}{\end{trivlist}}  \newenvironment{example}[1][Example]
 {\begin{trivlist} \item[\hskip \labelsep {\bfseries #1}]}{\end{trivlist}} 
\newcommand{\qed}{\nobreak \ifvmode \relax \else       \ifdim\lastskip<1.5em \hskip-\lastskip       \hskip1.5em plus0em minus0.5em \fi \nobreak       \vrule height0.75em width0.5em depth0.25em\fi}

\date{}

\makeatother

\usepackage{babel}
\begin{document}

\title{Connections between discriminants and the root distribution of polynomials
with rational generating function}

\author{Khang Tran\\
Department of Mathematics and Computer Science\\
Truman State University}
\maketitle
\begin{abstract}
Let $H_{m}(z)$ be a sequence of polynomials whose generating function
$\sum_{m=0}^{\infty}H_{m}(z)t^{m}$ is the reciprocal of a bivariate
polynomial $D(t,z)$. We show that in the three cases $D(t,z)=1+B(z)t+A(z)t^{2}$,
$D(t,z)=1+B(z)t+A(z)t^{3}$ and $D(t,z)=1+B(z)t+A(z)t^{4}$, where
$A(z)$ and $B(z)$ are any polynomials in $z$ with complex coefficients,
the roots of $H_{m}(z)$ lie on a portion of a real algebraic curve
whose equation is explicitly given. The proofs involve the $q$-analogue
of the discriminant, a concept introduced by Mourad Ismail. 
\end{abstract}
\footnote{The author acknowledges support from NSF grant DMS-0838434 \textquotedblright{}EMSW21MCTP:
Research Experience for Graduate Students\textquotedblright{} from
the University of Illinois at Urbana-Champaign.%
}

\section{Introduction}

In this paper we study the root distribution of a sequence of polynomials
satisfying one of the following three-term recurrences:
\begin{eqnarray*}
H_{m}(z)+B(z)H_{m-1}(z)+A(z)H_{m-2}(z) & = & 0,\\
H_{m}(z)+B(z)H_{m-1}(z)+A(z)H_{m-3}(z) & = & 0,\\
H_{m}(z)+B(z)H_{m-1}(z)+A(z)H_{m-4}(z) & = & 0,
\end{eqnarray*}
with certain initial conditions and $A(z),B(z)$ polynomials in $z$
with complex coefficients. For the study of the root distribution
of other sequences of polynomials that satisfy three-term recurrences,
see \cite{cc} and \cite{hs}. In particular, we choose the initial
conditions so that the generating function is 
\[
\sum_{m=0}^{\infty}H_{m}(z)t^{m}=\frac{1}{D(t,z)}
\]
where $D(t,z)=1+B(z)t+A(z)t^{2}$, $D(t,z)=1+B(z)t+A(z)t^{3}$, or
$D(t,z)=1+B(z)t+A(z)t^{4}$. We notice that the root distribution
of $H_{m}(z)$ will be the same if we replace $1$ in the numerator
by any monomial $N(t,z)$. If $N(t,z)$ is not a monomial, the root
distribution will be different. The quadratic case $D(t,z)=1+B(z)t+A(z)t^{2}$
is not difficult and it is also mentioned in \cite{tz}. We present
this case in Section 2 because it gives some directions to our main
cases, the cubic and quartic denominators $D(t,z)$, in Sections 3
and 4. 

Our approach uses the concept of $q$-analogue of the discriminant
($q$-discriminant) introduced by Ismail \cite{ismail}. The $q$-discriminant
of a polynomial $P_{n}(x)$ of degree $n$ and leading coefficient
$p$ is 
\begin{equation}
\mathrm{Disc}_{x}(P;q)=p^{2n-2}q^{n(n-1)/2}\prod_{1\le i<j\le n}(q^{-1/2}x_{i}-q^{1/2}x_{j})(q^{1/2}x_{i}-q^{-1/2}x_{j})\label{eq:qdisc}
\end{equation}
where $x_{i}$, $1\le i\le n,$ are the roots of $P_{n}(x)$. This
$q$-discriminant is $0$ if and only if a quotient of roots $x_{i}/x_{j}$
equals $q$. As $q\rightarrow1$, this $q$-discriminant becomes the
ordinary discriminant which is denoted by $\mathrm{Disc}_{x}P(x)$.
For the study of resultants and ordinary discriminants and their various
formulas, see \cite{aar}, \cite{apostal}, \cite{dilcherstolarsky},
and \cite{gisheismail}.

We will see that the concept of $q$-discriminant is useful in proving
connections between the root distribution of a sequence of polynomials
$H_{m}(z)$ and the discriminant of the denominator of its generating
function $\mathrm{Disc}_{t}D(t,z)$. We will show in the three cases
mentioned above that the roots of $H_{m}(z)$ lie on a portion of
a real algebraic curve (see Theorem \ref{quadratic}, Theorem \ref{cubic},
and Theorem \ref{quartic}). For the study of sequences of polynomials
whose roots approach fixed curves, see \cite{boyergoh,boyergoh-1,boyergoh-2}.
Other studies of the limits of zeros of polynomials satisfying a linear
homogeneous recursion whose coefficients are polynomials in $z$ are
given in \cite{bkw,bkw-1}. The $q$-discriminant will appear as the
quotient $q$ of roots in $t$ of $D(t,z)$. One advantage of looking
at the quotients of roots is that, at least in the three cases above,
although the roots of $H_{m}(z)$ lie on a curve depending on $A(z)$
and $B(z)$, the quotients of roots $t=t(z)$ of $D(t,z)$ lie on
a fixed curve independent of these two polynomials. We will show that
this independent curve is the unit circle in the quadratic case and
two peculiar curves (see Figures 1 and 2 in Sections 3 and 4) in the
cubic and quartic cases. From computer experiments, this curve looks
more complicated in the quintic case $D(z,t)=1+B(z)t+A(z)t^{5}$ (see
Figure 3 in Section 4).

As an application of these theorems, we will consider an example where
$D(t,z)=1+(z^{2}-2z+a)t+z^{2}t^{2}$ and $a\in\mathbb{R}$. We will
see that the roots of $H_{m}(z)$ lie either on portions of the circle
of radius $\sqrt{a}$ or real intervals depending on the value $a$
compared to the critical values $0$ and $4$. Also, the endpoints
of the curves where the roots of $H_{m}(z)$ lie are roots of $\mathrm{Disc}_{t}D(t,z)$.
Interestingly, the critical values $0$ and $4$ are roots of the
double discriminant $\mathrm{Disc}_{z}\mathrm{Disc}_{t}D(t,z)=4096a^{3}(a-4).$

\section{The quadratic denominator}

In this section, we will consider the root distribution of $H_{m}(z)$
when the denominator of the generating function is $D(t,z)=1+B(z)t+A(z)t^{2}$. 

\begin{theorem}\label{quadratic}

Let $H_{m}(z)$ be a sequence of polynomials whose generating function
is
\[
\sum H_{m}(z)t^{m}=\frac{1}{1+B(z)t+A(z)t^{2}}
\]
where $A(z)$ and $B(z)$ are polynomials in $z$ with complex coefficients.
The roots of $H_{m}(z)$ which satisfy $A(z)\ne0$ lie on the curve
$\mathcal{C}_{2}$ defined by
\[
\Im\frac{B^{2}(z)}{A(z)}=0\qquad\mbox{and}\qquad0\le\Re\frac{B^{2}(z)}{A(z)}\le4,
\]
and are dense there as $m\rightarrow\infty$. 

\end{theorem}

\begin{proof}

Suppose $z_{0}$ is a root of $H_{m}(z)$ which satisfies $A(z_{0})\ne0$.
Let $t_{1}=t_{1}(z_{0})$ and $t_{2}=t_{2}(z_{0})$ be the roots of
$D(t,z_{0})$. If $t_{1}=t_{2}$ then $\mathrm{Disc}_{t}D(t,z_{0})=B^{2}(z_{0})-4A(z_{0})=0$.
In this case $z_{0}$ belongs to $\mathcal{C}_{2}$, and we only need
to consider the case $t_{1}\ne t_{2}$. By partial fractions, we have
\begin{eqnarray}
\frac{1}{D(t,z_{0})} & = & \frac{1}{A(z_{0})(t-t_{1})(t-t_{2})}\nonumber \\
 & = & \frac{1}{A(z_{0})(t_{1}-t_{2})}\left(\frac{1}{t-t_{1}}-\frac{1}{t-t_{2}}\right)\nonumber \\
 & = & \frac{1}{A(z_{0})}\sum_{m=0}^{\infty}\frac{t_{1}^{m+1}-t_{2}^{m+1}}{(t_{1}-t_{2})t_{1}^{m+1}t_{2}^{m+1}}t^{n}.\label{eq:quadraticT}
\end{eqnarray}
Thus if we let $t_{1}=qt_{2}$ then $q$ is an $(m+1)$-st root of
unity and $q\ne1$. By the definition of $q$-discriminant in \eqref{eq:qdisc},
$q$ is a root of $\mathrm{Disc}_{t}(D(t,z_{0});q)$ which equals
\[
q\left(B^{2}(z_{0})-(q+q^{-1}+2)A(z_{0})\right).
\]
This implies that 
\[
\frac{B^{2}(z_{0})}{A(z_{0})}=q+q^{-1}+2.
\]
Thus $z_{0}\in\mathcal{C}_{2}$ since $q$ is an $(m+1)$-th root
of unity. 

The map $B^{2}(z)/A(z)$ maps an open neighborhood $U$ of a point
on $\mathcal{C}_{2}$ onto an open set which contains a point $2\Re q+2$,
where $q$ is an $(m+1)$-th root of unity, when $m$ is large. From
\eqref{eq:quadraticT}, there is a solution of $H_{m}(z)$ in $U$.
The density of the roots of $H_{m}(z)$ follows.

\end{proof}

\begin{example}

We consider an example in which the generating function of $H_{m}(z)$
is given by 
\[
\frac{1}{z^{2}t^{2}+(z^{2}-2z+a)t+1}=\sum_{m=0}^{\infty}H_{m}(z)t^{m}
\]
where $a\in\mathbb{R}$ . Let $z=x+iy$. We exhibit the three possible
cases for the root distribution of $H_{m}(z)$ depending on $a$:
\begin{enumerate}
\item If $a\le0$, the roots of $H_{m}(z)$ lie on the two real intervals
defined by 
\[
(x^{2}+a)(x^{2}-4x+a)\le0.
\]

\item If $0<a\le4$, the roots of $H_{m}(z)$ can lie either on the half
circle $x^{2}+y^{2}=a$, $x\ge0$, or on the real interval defined
by $x^{2}-4x+a\le0$.
\item If $a>4$, the roots of $H_{m}(z)$ lie on two parts of the circle
$x^{2}+y^{2}=a$ restricted by $0\le x\le2$.
\end{enumerate}
Indeed, by complex expansion, we have
\[
\Im\frac{B^{2}(z)}{A(z)}=\frac{2y(x^{2}+y^{2}-a)P}{(x^{2}+y^{2})^{2}}\qquad\mbox{and}\qquad\mbox{\ensuremath{\Re}}\frac{B^{2}(z)}{A(z)}=\frac{P^{2}-Q^{2}}{(x^{2}+y^{2})^{2}},
\]
where
\[
P=ax-2x^{2}+x^{3}-2y^{2}+xy^{2}\qquad\mbox{and}\qquad Q=y(x^{2}+y^{2}-a).
\]

Theorem \ref{quadratic} yields three cases: $y=0$, $x^{2}+y^{2}-a=0$
or $P=0$. Since $\Re\left(B^{2}(z)/A(z)\right)\ge0$, all these cases
give $Q=0$. We note that if $x^{2}+y^{2}-a=0$ then the condition
$\Re\left(B^{2}(z)/A(z)\right)\le4$ reduces to 
\begin{equation}
x(a+x^{2}+y^{2})(ax-4x^{2}+x^{3}-4y^{2}+xy^{2})=4a^{2}x(x-2)\le0.\label{eq:circle}
\end{equation}
Suppose $a\le0$. Then the condition $Q=0$ implies that the roots
of $H_{m}(z)$ are real. The condition $\Re\left(B^{2}(z)/A(z)\right)\le4$
becomes 
\begin{equation}
(x^{3}-2x^{2}+ax)^{2}-4x^{4}=x^{2}(x^{2}+a)(x^{2}-4x+a)\le0.\label{eq:realline}
\end{equation}
Suppose $0<a\le4$. The roots of $H_{m}(z)$ lie either on the half
circle $x^{2}+y^{2}-a=0$, $x\ge0$ (from the inequality \eqref{eq:circle}),
or on the real interval given by $x^{2}-4x+a\le0$ (from the inequality
\eqref{eq:realline}). If $a>4$ then the roots of $H_{m}(z)$ lie
on the two parts of the circle $x^{2}+y^{2}-a=0$ restricted by $0\le x\le2$
(from the inequality \eqref{eq:circle}).

We notice that in this example, the inequality $\Re\left(B^{2}(z)/A(z)\right)\le4$
gives the endpoints of the curves where the roots of $H_{m}(z)$ lie.
Thus, these endpoints are roots of $\mathrm{Disc}_{t}(1+B(z)t+A(z)t^{2})=B^{2}(z)-4A(z)$.
Moreover the critical values of $a$, which are $0$ and $4$, are
roots of the double discriminant of the denominator
\[
\mathrm{Disc}_{z}\mathrm{Disc}_{t}(1+(z^{2}-2z+a)t+z^{2}t^{2})=4096a^{3}(a-4).
\]
This comes from the fact that the endpoints of the fixed curves containing
the roots of $H_{m}(z)$ are the roots of $\mathrm{Disc}_{t}(1+(z^{2}-2z+a)t+z^{2}t^{2})$.
When this discriminant has a double root as a polynomial in $z$,
some two endpoints of the fixed curves coincide. That explains the
change in the shape of the root distribution.

\end{example}

\section{The cubic denominator}

In this section we show that in the cubic case $D(t,z)=1+B(z)t+A(z)t^{3}$,
the roots of $H_{m}(z)$ lie on a portion of a real algebraic curve.
As we see in the proof of Theorem \ref{quadratic}, we can first consider
the distribution of the quotients of roots $q=t_{i}/t_{j}$ of $D(t,z)$,
and then we can relate to the root distribution of $H_{m}(z)$ using
the $q$-discriminant. While in the previous section this quotient
lies on the unit circle, in this section we show that this quotient
lie on the curve in Figure 1. 

\begin{lemma}\label{cubicqlemma}

Suppose $\zeta_{1},\zeta_{2}\ne0$ are complex numbers such that $1/\zeta_{1}+1/\zeta_{2}+1=0$
and 
\begin{equation}
\frac{\zeta_{1}^{m+1}-1}{\zeta_{1}-1}=\frac{\zeta_{2}^{m+1}-1}{\zeta_{2}-1}.\label{eq:cubic-q}
\end{equation}
Then $\zeta_{1}$ and $\zeta_{2}$ lie on the union $C_{1}\cup C_{2}\cup C_{3}$
where the Cartesian equations of $C_{1}$, $C_{2}$ and $C_{3}$ are
given by
\begin{eqnarray*}
C_{1} & : & (x+1)^{2}+y^{2}=1,x\le-\frac{1}{2},\\
C_{2} & : & x=-\frac{1}{2},-\frac{\sqrt{3}}{2}\le y\le\frac{\sqrt{3}}{2},\\
C_{3} & : & x^{2}+y^{2}=1,x\ge-\frac{1}{2},
\end{eqnarray*}
and are dense there as $m\rightarrow\infty$. 

\end{lemma}

\begin{figure}[H]
\begin{centering}
\includegraphics{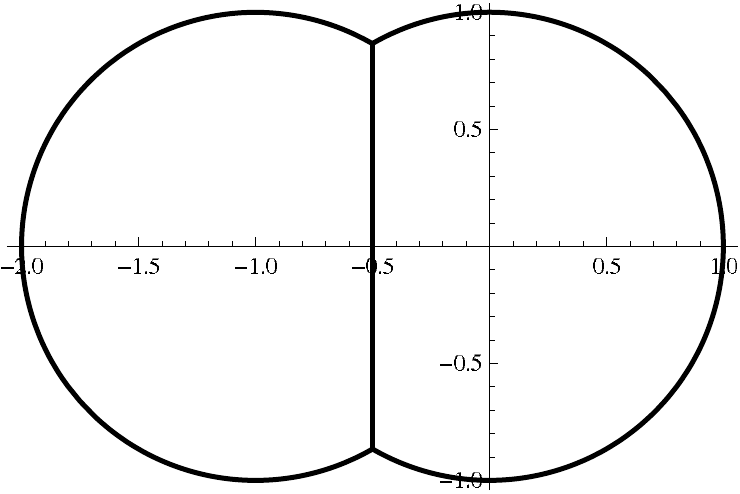}
\par\end{centering}

\caption{Distribution of the quotients of the roots of the cubic denominator}
\label{quotientcubicfig}
\end{figure}

\begin{proof}

We can rewrite \eqref{eq:cubic-q} as 
\[
\sum_{k=0}^{m}\zeta_{1}^{k}=\sum_{k=0}^{m}\zeta_{2}^{k}
\]
where we can replace $\zeta_{2}$ by $-\zeta_{1}/(\zeta_{1}+1)$.
By multiplying both sides by $(\zeta_{1}+1)^{m}$, we note that there
are at most $2m-2$ solutions $\zeta=\zeta_{1}\ne0,-2$ counting multiplicity.
Let $m=3n+k$ where $k=1,2,3$. From implicit differentiation, we
can check that the equation \eqref{eq:cubic-q} has roots at $e^{2\pi i/3},e^{4\pi i/3}$
with multiplicity $k-1$. After subtracting this number of roots from
$2m-2$, we conclude that there are at most $6n$ roots $\zeta\ne0,-2,e^{2\pi i/3},e^{4\pi i/3}$.
We first show that if $\zeta\ne-2$ is a root, then so is $-\zeta-1$.
From the two equations in the hypothesis, we note that $\zeta\ne0,-1$
and 
\[
\sum_{k=0}^{m}\zeta^{k}=\sum_{k=0}^{m}\left(-\frac{\zeta}{\zeta+1}\right)^{k}.
\]
Subtracting $1$, then dividing by $\zeta$ and multiplying both sides
by $(\zeta+1)^{m}$ , we obtain
\begin{eqnarray*}
0 & = & \sum_{k=0}^{m-1}\zeta^{k}(\zeta+1)^{m}+\sum_{k=0}^{m-1}(\zeta+1)^{m-k-1}(-\zeta)^{k}\\
 & = & \sum_{k=0}^{m-1}\zeta^{k}(\zeta+1)^{m-k-1}\left((\zeta+1)^{k+1}-(-1)^{k+1}\right)\\
 & = & (\zeta+2)\sum_{k=0}^{m-1}\zeta^{k}(\zeta+1)^{m-k-1}\sum_{i=0}^{k}(\zeta+1)^{k-i}(-1)^{i}\\
 & = & (\zeta+2)\sum_{k=0}^{m-1}\sum_{i=0}^{k}\zeta^{k}(-\zeta-1)^{m-1-i}.
\end{eqnarray*}
By interchanging the summation and reversing the index of summation
we obtain
\begin{eqnarray*}
\sum_{k=0}^{m-1}\sum_{i=0}^{k}\zeta^{k}(-\zeta-1)^{m-1-i} & = & \sum_{i=0}^{m-1}\sum_{k=i}^{m-1}\zeta^{k}(-\zeta-1)^{m-1-i}\\
 & = & \sum_{i=0}^{m-1}\sum_{k=0}^{i}\zeta^{m-1-k}(-\zeta-1)^{i}.
\end{eqnarray*}
Hence we have symmetry between $\zeta$ and $-1-\zeta$ in the two
double summations.

Our goal is to show that the number of roots $\zeta\ne0,-2,e^{2\pi i/3},e^{4\pi i/3}$
on $C_{1}\cup C_{2}\cup C_{3}$ is at least $6n$, counting multiplicities.
Then all roots will lie on $C_{1}\cup C_{2}\cup C_{3}$ since we have
at most $6n$ roots $\zeta\ne0,-2,e^{2\pi i/3},e^{4\pi i/3}$. By
the symmetry of roots mentioned above, if $\zeta\ne-2$ is a solution
in $C_{1}$ then $(-1-1/\zeta,-\zeta-1)$ is a solution in $C_{2}\times C_{3}$.
Hence there is a bijection between roots in $C_{1}\backslash\{-2\}$,
$C_{2}$ and $C_{3}$. Thus if $C_{1}\backslash\{e^{2\pi i/3},e^{4\pi i/3}\}$
contains at least $2n+1$ roots then all of the roots lie on $C_{1}\cup C_{2}\cup C_{3}$.
Let $\zeta=\zeta_{1}$ be a root on $C_{1}\backslash\{e^{2\pi i/3},e^{4\pi i/3}\}$.
Then the equation $1/\zeta_{1}+1/\zeta_{2}+1=0$ gives $\zeta_{2}=\bar{\zeta}$.
Thus \eqref{eq:cubic-q} gives 
\[
\Im\frac{\zeta^{m+1}-1}{\zeta-1}=0.
\]
Write $\zeta=re^{i\theta}$ where $r=-2\cos\theta$, $\cos\theta\le-1/2$.
Then complex expansion yields
\[
r^{m+2}\sin m\theta-r^{m+1}\sin(m+1)\theta+r\sin\theta=0.
\]
Divide $r$, replace $r$ by $-2\cos\theta$ and combine the first
two terms to obtain
\begin{eqnarray*}
0 & = & (-1)^{m+1}2^{m}\cos^{m}\theta\left(2\sin m\theta\cos\theta+\sin(m+1)\theta\right)+\sin\theta\\
 & = & (-1)^{m+1}2^{m}\cos^{m}\theta\left(2\sin(m+1)\theta-2\cos m\theta\sin\theta+\sin(m+1)\theta\right)+\sin\theta\\
 & = & (-1)^{m+1}2^{m}\cos^{m}\theta\left(2\sin(m+1)\theta+\sin m\theta\cos\theta-\cos m\theta\sin\theta\right)+\sin\theta\\
 & = & (-1)^{m+1}2^{m}\cos^{m}\theta\left(2\sin(m+1)\theta+\sin(m-1)\theta\right)+\sin\theta.
\end{eqnarray*}
We note that the right side has different signs if $\sin(m+1)\theta=1$
and $\sin(m+1)\theta=-1$. Thus we can apply the Intermediate Value
Theorem on several intervals whose boundaries are the solutions of
$\sin(m+1)\theta=\pm1$. The equations $\sin(m+1)\theta=\pm1$ give
\[
(m+1)\theta=\pm\frac{\pi}{2}+2j\pi.
\]
The condition $2\pi/3<\theta<4\pi/3$ and the fact that $m=3n+k$,
$k=1,2,3$, yield
\[
n+\frac{k+1}{3}\pm\frac{1}{4}<j<2n+\frac{2(k+1)}{3}\pm\frac{1}{4}.
\]
If $k=1$, we have at least $2n+1$ roots coming from $2n+1$ intervals
formed by the $2n+2$ points 
\[
\frac{2j\pi\pm\pi/2}{m+1},
\]
where $n<j\le2n+1$. If $k=2$, we have at least $2n+1$ roots coming
from $2n+1$ intervals formed by the $2n+2$ points
\[
\left\{ \frac{2j-\pi/2}{m+1}:n+1\le j<2n+2\right\} \cup\left\{ \frac{2j+\pi/2}{m+1}:n+1<j\le2n+2\right\} .
\]
If $k=3$, we have at least $2n+1$ roots coming from $2n+1$ intervals
formed by the $2n+2$ points 
\[
\frac{2j\pi\pm\pi/2}{m+1},
\]
where $n+1<j<2n+2$. The density follows from the distribution of
$2n+1$ roots mentioned above. The lemma follows.

\end{proof}

\begin{theorem}\label{cubic}

Let $H_{m}(z)$ be a sequence of polynomials whose generating function
is
\[
\sum H_{m}(z)t^{m}=\frac{1}{1+B(z)t+A(z)t^{3}}
\]
where $A(z)$ and $B(z)$ are polynomials in $z$ with complex coefficients.
The roots of $H_{m}(z)$ which satisfy $A(z)\ne0$ lie on the curve
$\mathcal{C}_{3}$ defined by
\[
\Im\frac{B^{3}(z)}{A(z)}=0\qquad\mbox{and}\qquad0\le-\Re\frac{B^{3}(z)}{A(z)}\le\frac{3^{3}}{2^{2}},
\]
and are dense there as $m\rightarrow\infty$.

\end{theorem}

\begin{proof}

For a little simplification, we consider the roots of $H_{m-1}(z)$.
Let $z_{0}$ be a root of $H_{m-1}(z)$ which satisfies $A(z_{0})\ne0$.
Let $t_{1}=t_{1}(z_{0})$, $t_{2}=t_{2}(z_{0})$ and $t_{3}=t_{3}(z_{0})$
be the roots of $D(t,z_{0})=1+B(z_{0})t+A(z_{0})t^{3}$. It suffices
to consider $\mathrm{Disc}_{t}(D(t,z_{0}))=-4A(z_{0})B^{3}(z_{0})-27A^{2}(z_{0})\ne0$.
By partial fractions, the function $1/D(t,z_{0})$ is
\[
\frac{1}{A(z_{0})(t_{1}-t_{2})(t_{1}-t_{3})(t-t_{1})}+\frac{1}{A(z_{0})(t_{2}-t_{1})(t_{2}-t_{3})(t-t_{2})}+\frac{1}{A(z_{0})(t_{3}-t_{1})(t_{3}-t_{2})(t-t_{3})}.
\]
We expand $1/(t-t_{i})$ using geometric series and write the expression
above as 
\[
\sum_{m=1}^{\infty}\frac{t_{1}^{m+1}t_{2}^{m}-t_{1}^{m}t_{2}^{m+1}-t_{1}^{m+1}t_{3}^{m}+t_{2}^{m+1}t_{3}^{m}+t_{1}^{m}t_{3}^{m+1}-t_{2}^{m}t_{3}^{m+1}}{A(z_{0})t_{1}^{m}t_{2}^{m}t_{3}^{m}(t_{1}-t_{2})(t_{1}-t_{3})(t_{2}-t_{3})}t^{m-1}.
\]
Since $z_{0}$ is a root of $H_{m-1}(z)$, we have 
\[
t_{1}^{m+1}t_{2}^{m}-t_{1}^{m}t_{2}^{m+1}-t_{1}^{m+1}t_{3}^{m}+t_{2}^{m+1}t_{3}^{m}+t_{1}^{m}t_{3}^{m+1}-t_{2}^{m}t_{3}^{m+1}=0.
\]
We divide this equation by $t_{3}^{2m+1}$ and let $q=q_{1}=t_{1}/t_{3}$,
$q_{2}=t_{2}/t_{3}$ to obtain 
\[
q_{1}^{m+1}q_{2}^{m}-q_{1}^{m}q_{2}^{m+1}-q_{1}^{m+1}+q_{2}^{m+1}+q_{1}^{m}-q_{2}^{m}=0
\]
 where $q_{1}+q_{2}+1=0$ since $t_{1}+t_{2}+t_{2}=0$. The equation
can be written as 
\[
q_{1}^{m}q_{2}^{m}(q_{1}-q_{2})-q_{1}^{m}(q_{1}-1)+q_{2}^{m}(q_{2}-1)=0.
\]
Since $q_{1}^{m}q_{2}^{m}(q_{1}-q_{2})=q_{1}^{m}q_{2}^{m}(q_{1}-1)-q_{1}^{m}q_{2}^{m}(q_{2}-1)$
and $q_{1},q_{2}\ne0,1$, this equation becomes
\[
\frac{q_{1}^{m}-1}{q_{1}^{m}(q_{1}-1)}=\frac{q_{2}^{m}-1}{q_{2}^{m}(q_{2}-1)}.
\]
Let $\zeta_{1}=1/q_{1}$ and $\zeta_{2}=1/q_{2}$ and add 1 to both
sides. Then 
\[
\frac{\zeta_{1}^{m+1}-1}{\zeta_{1}-1}=\frac{\zeta_{2}^{m+1}-1}{\zeta_{2}-1}.
\]
Thus $\zeta_{1}$ and $\zeta_{2}$ (and also $q_{1}$ and $q_{2}$)
lie on the curve given in Lemma \ref{cubicqlemma}. Since $q_{1}$
and $q_{2}$ are given by quotients of two roots, they are roots of
the $q$-discriminant given by
\[
\mbox{Disc}_{t}(D(t,z_{0});q)=-B^{3}(z_{0})A(z_{0})q^{2}(1+q)^{2}-A^{2}(z_{0})(1+q+q^{2})^{3}.
\]
This gives 
\[
\frac{B^{3}(z_{0})}{A(z_{0})}=-\frac{(1+q+q^{2})^{3}}{q^{2}(1+q)^{2}}.
\]
It remains to show that the map
\[
f(q)=-\frac{(1+q+q^{2})^{3}}{q^{2}(1+q)^{2}}
\]
maps the curve in Figure \ref{quotientcubicfig} to the real interval
$[-27/4,0]$. Let $q$ be a point on the this curve. We note that
\[
f(q)=f(-1-q)=-\frac{(q^{-1}+1+q)^{3}}{q^{-1}+2+q}.
\]
Since $q$ lies on the curve in Figure \ref{quotientcubicfig}, we
have the three possible cases $\bar{q}=-1-q$, $|q|=1$ or $|-1-q|=1$.
In the first case, $\Im f(q)=0$ since $f(q)=\overline{f(q)}$. In
the second and third cases, $\Im f(q)=0$ since $q+q^{-1}\in\mathbb{R}$
and $f(q)=f(-1-q)$. Furthermore, $f(q)$ attains its minimum and
maximum when $q=1$ and $q=e^{2\pi i/3}$ respectively. The density
of the roots of $H_{m}(z)$ follows from similar arguments as in the
proof of Theorem \ref{quadratic}. 

\end{proof}

\section{The quartic denominator}

In this section, we will show that in the case $D(t,z)=1+B(z)+A(z)t^{4}$
the roots of $H_{m}(z)$ lie on a portion of a real algebraic curve.
Similar to the approach in the previous sections, we first consider
the distribution of the quotients of roots of $D(t,z)$. Before looking
at these quotients, let us recall that the Chebyshev polynomial of
the second kind $U_{m}(z)$ is 
\[
U_{m}(z)=\frac{\sin(m+1)\theta}{\sin\theta}
\]
where 
\[
z=\cos\theta.
\]
Suppose $z_{1},z_{2}\in\mathbb{C}$ such that $|z_{1}|=|z_{2}|$.
Let $e^{2i\theta}=z_{1}/z_{2}$ and $z=\cos\theta$. If $k$ is a
positive integer then 
\begin{eqnarray}
\frac{z_{1}^{k}-z_{2}^{k}}{z_{1}-z_{2}} & = & (z_{1}z_{2})^{(k-1)/2}\frac{(z_{1}/z_{2})^{m/2}-(z_{2}/z_{1})^{m/2}}{(z_{1}/z_{2})^{1/2}-(z_{2}/z_{1})^{1/2}}\nonumber \\
 & = & (z_{1}z_{2})^{(k-1)/2}U_{m}(z).\label{eq:ChebyshevU}
\end{eqnarray}
By analytic continuation, we can extend this identity to any pair
of complex numbers $z_{1}$ and $z_{2}$ with 
\[
2z=\left(\frac{z_{1}}{z_{2}}\right)^{1/2}+\left(\frac{z_{2}}{z_{1}}\right)^{1/2}.
\]

\begin{lemma}\label{quarticqlemma}

Suppose $z_{0}$ is a root of $H_{m}(z)$ and $q=q(z_{0})$ is a quotient
of two roots in $t$ of $1+B(z_{0})t+A(z_{0})t^{4}$. Then the set
of all such quotients belongs to the curve depicted in Figure 2, where
the Cartesian equation of the quartic curve on the left is 
\[
1+2x+2x^{2}+2x^{3}+x^{4}-2y^{2}+2xy^{2}+2x^{2}y^{2}+y^{4}=0,
\]
and the curve on the right is the unit circle with real part at least
$-1/3$. All such quotients are dense on this curve as $m\rightarrow\infty$.

\end{lemma}

\begin{figure}[H]
\begin{centering}
\includegraphics{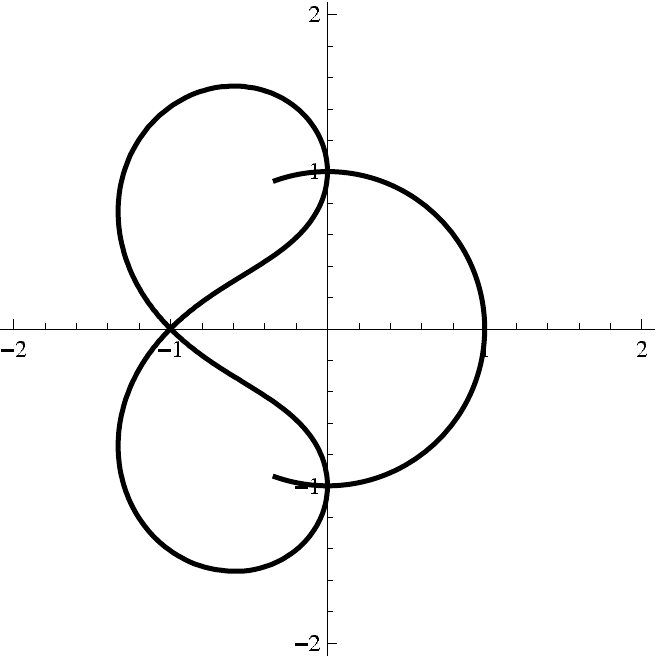}
\par\end{centering}

\caption{Distribution of the quotients of the roots of the quartic denominator}
\label{quotientquarticfig}
\end{figure}

\begin{proof}

For each $z_{0}\in\mathbb{C}$, let $t_{1}=t_{1}(z_{0})$, $t_{2}=t_{2}(z_{0})$,
$t_{3}=t_{3}(z_{0})$, and $t_{4}=t_{4}(z_{0})$ be the roots of the
denominator $1+B(z_{0})t+A(z_{0})t^{4}$. By partial fractions, we
have 
\begin{eqnarray*}
\frac{1}{1+B(z_{0})t+A(z_{0})t^{4}} & = & \frac{1}{A(z_{0})(t-t_{1})(t-t_{2})(t-t_{3})(1-t_{4})}\\
 & = & \sum_{m=0}^{\infty}H_{m}(z_{0})t^{m},
\end{eqnarray*}
where 
\begin{eqnarray*}
A(z_{0})H_{m}(z_{0}) & = & \frac{1}{t_{1}^{m+1}(t_{1}-t_{2})(t_{1}-t_{3})(t_{1}-t_{4})}+\frac{1}{t_{2}^{m+1}(t_{2}-t_{1})(t_{2}-t_{3})(t_{2}-t_{4})}\\
 &  & +\frac{1}{t_{3}^{m+1}(t_{3}-t_{1})(t_{3}-t_{2})(t_{3}-t_{4})}+\frac{1}{t_{4}^{m+1}(t_{4}-t_{1})(t_{4}-t_{2})(t_{4}-t_{3})}.
\end{eqnarray*}
 Let $q_{1}=t_{1}/t_{4}$, $q_{2}=t_{2}/t_{4}$, $q_{3}=t_{3}/t_{4}$.
For a little reduction in the powers of $q_{i}$, $1\le i\le3$, we
will consider the roots of the polynomial $H_{m-2}(z)$. We put all
terms of $A(z_{0})H_{m-2}(z_{0})$ over a common denominator and then
divide the numerator by $t_{4}^{3m}$. The condition $H_{m-2}(z_{0})=0$
implies 
\begin{eqnarray}
0 & = & q_{1}^{m+1}(-q_{2}^{m-1}q_{3}^{m-1}(q_{2}-q_{3})+q_{2}^{m}-q_{3}^{m}-q_{2}^{m-1}+q_{3}^{m-1})\nonumber \\
 &  & +q_{1}^{m}(q_{2}^{m-1}q_{3}^{m-1}(q_{2}^{2}-q_{3}^{2})-q_{2}^{m+1}+q_{3}^{m+1}+q_{2}^{m-1}-q_{3}^{m-1})\nonumber \\
 &  & +q_{1}^{m-1}(-q_{2}^{m}q_{3}^{m}(q_{2}-q_{3})+q_{2}^{m+1}-q_{3}^{m+1}-q_{2}^{m}+q_{3}^{m})\nonumber \\
 &  & +q_{2}^{m-1}q_{3}^{m-1}(q_{2}-q_{3})-q_{2}^{m-1}q_{3}^{m-1}(q_{2}^{2}-q_{3}^{2})+q_{2}^{m}q_{3}^{m}(q_{2}-q_{3}).\label{eq:quartic-q}
\end{eqnarray}
The fact 
\begin{eqnarray*}
t_{1}+t_{2}+t_{3}+t_{4} & = & 0\\
t_{1}t_{2}+t_{1}t_{3}+t_{1}t_{4}+t_{2}t_{3}+t_{2}t_{4}+t_{3}t_{4} & = & 0
\end{eqnarray*}
gives
\[
q_{2}+q_{3}=-1-q_{1}\qquad\mbox{and}\qquad q_{2}q_{3}=q_{1}^{2}+q_{1}+1.
\]

From the symmetric reductions, the right side of \eqref{eq:quartic-q},
after being divided by $q_{2}-q_{3}$, is a polynomial in $q_{1}$
of degree $3m-1$ . We used a computer algebra system to check for
the root distribution of this polynomial in the case $m\le5$. We
now assume that $m\ge6$. We will show that the number of roots $q_{1}$
lying on the two curves in Figure \ref{quotientquarticfig} is at
least $3m-1$. The first step is to show that if the set of $q_{1}$
belongs to the unit circle with $\Re q_{1}\ge-1/3$ and is dense there
as $m\rightarrow\infty$ then the set of $q_{2}$ and $q_{3}$ belongs
to the quartic curve given in the lemma and is dense on this quartic
curve as $m\rightarrow\infty$. Then we will find the number of roots
$q_{1}$ on the unit circle with $\Re q_{1}\ge-1/3$.

Suppose $q_{1}=e^{i\pi\theta}$ lies on the unit circle and $1\ge\cos\theta\ge-1/3$.
We note that $q_{2}$ and $q_{3}$ are the two roots of the equation
\[
f(q):=q^{2}+(1+q_{1})q+q_{1}^{2}+q_{1}+1=0.
\]
Thus the quadratic formula gives
\[
q=\frac{-1-e^{i\theta}\pm ie^{i\theta/2}\sqrt{6\cos\theta+2}}{2}.
\]
Splitting the real and imaginary parts of the function on the left
side, we leave it to the reader to check that this function maps the
interval $1\ge\cos\theta\ge-1/3$ to the quartic curve 
\[
1+2x+2x^{2}+2x^{3}+x^{4}-2y^{2}+2xy^{2}+2x^{2}y^{2}+y^{4}=0.
\]

We now compute the number of roots $q_{1}=e^{i\pi\theta}$ with $\cos\theta\ge-1/3$.
We first consider $q_{1}\ne\pm i,1$. Let 
\[
2\zeta=\left(\frac{q_{2}}{q_{3}}\right)^{1/2}+\left(\frac{q_{3}}{q_{2}}\right)^{1/2}.
\]
Equation \eqref{eq:ChebyshevU} gives 
\begin{eqnarray*}
\frac{q_{2}^{m}-q_{3}^{m}}{q_{2}-q_{3}} & = & (q_{2}q_{3})^{(m-1)/2}U_{m-1}\left(\zeta\right)
\end{eqnarray*}
where
\begin{equation}
\zeta^{2}=\frac{1}{4}\frac{(q_{2}+q_{3})^{2}}{q_{2}q_{3}}=\frac{(q_{1}+1)^{2}}{4(q_{1}^{2}+q_{1}+1)}=\frac{1}{4(2\cos\theta+1)}+\frac{1}{4}\in\mathbb{R}.\label{eq:zetaform}
\end{equation}
We divide \eqref{eq:quartic-q} by $q_{2}-q_{3}$ and rewrite it in
terms of Chebyshev polynomials:

\begin{eqnarray*}
0 & = & U_{m}(\zeta)(-q_{1}^{m}+q_{1}^{m-1})(q_{1}^{2}+q_{1}+1)^{m/2}\\
 &  & +U_{m-1}(\zeta)(q_{1}^{m+1}-q_{1}^{m-1})(q_{1}^{2}+q_{1}+1)^{(m-1)/2}\\
 &  & +U_{m-2}(\zeta)(-q_{1}^{m+1}+q_{1}^{m})(q_{1}^{2}+q_{1}+1)^{(m-2)/2}\\
 &  & +(q_{1}^{2}+q_{1}+1)^{m-1}(-3q_{1}^{m+1}-2q_{1}^{m}-q_{1}^{m-1}+q_{1}^{2}+2q_{1}+3).
\end{eqnarray*}
We divide this equation by $q_{1}^{(3m-2)/2}(1-q_{1})(q_{1}+q_{1}^{-1}+1)^{m/2}$
and write $(q_{1}^{m}-1)/(q_{1}-1)$ in terms of Chebyshev polynomials.
We obtain 
\begin{eqnarray*}
0 & = & U_{m}(\zeta)+2\zeta U_{m-1}(\zeta)+U_{m-2}(\zeta)/(q_{1}+q_{1}^{-1}+1)\\
 &  & +(q_{1}+q_{1}^{-1}+1)^{m/2-1}\left(3U_{m}\left(\xi\right)+2U_{m-2}(\xi)+U_{m-4}(\xi)\right),
\end{eqnarray*}
where
\begin{equation}
\xi^{2}=\frac{(q_{1}+1)^{2}}{4q_{1}}=\frac{2\cos\theta+2}{4}.\label{eq:xiform}
\end{equation}
 Finally, from \eqref{eq:zetaform} we can replace $1/(q_{1}+q_{1}^{-1}+1)$
by $(4\zeta^{2}-1)$ and use the recurrence definition of the Chebyshev
polynomials to rewrite this equation in the symmetric form below:
\begin{eqnarray}
0 & = & (4\zeta^{2}-1)^{(m-2)/4}\left(3U_{m}\left(\zeta\right)+2U_{m-2}(\zeta)+U_{m-4}(\zeta)\right)\nonumber \\
 &  & +(4\xi^{2}-1)^{(m-2)/4}\left(3U_{m}\left(\xi\right)+2U_{m-2}(\xi)+U_{m-4}(\xi)\right).\label{eq:quartic-real}
\end{eqnarray}

From this symmetric form, the right expression remains the same if
we interchange $\zeta$ and $\xi$ or if we interchange $\cos\theta$
and $-\cos\theta/(2\cos\theta+1)$ (from \eqref{eq:zetaform} and
\eqref{eq:xiform}). Thus the numbers of roots $q_{1}$ are the same
in the two cases $0<\cos\theta<1$ and $-1/3<\cos\theta<0$. It is
sufficient to count the number of roots $0<\cos\theta<1$ or $1/2<\xi^{2}<1$.
Let $\cos\alpha=\xi$ and $U_{m}(\xi)=\sin(m+1)\alpha/\sin\alpha$
where $-\pi/4<\alpha<\pi/4$, $\alpha\ne0$. The idea is to show that
in this case the summand 
\begin{equation}
(4\xi^{2}-1)^{(m-2)/4}\left(3U_{m}\left(\xi\right)+2U_{m-2}(\xi)+U_{m-4}(\xi)\right)\label{eq:dominantterm}
\end{equation}
dominates the right expression of \eqref{eq:quartic-real}. Since
$\zeta^{2}$ and $\xi^{2}$ in \eqref{eq:quartic-real} are real numbers
and the Chebyshev polynomials in this equation are either even or
odd, we can apply the Intermediate Value Theorem. We note that \eqref{eq:dominantterm}
has different signs when $\sin(m+1)\alpha=1$ and when $\sin(m+1)\alpha=-1$.
Suppose $\sin(m+1)\alpha=\pm1$ and $-\pi/4<\alpha<\pi/4$. Since
\[
4\zeta^{2}-1=\frac{1}{1+2\cos\theta}<1,
\]
it suffices to show
\begin{equation}
\left|(4\xi^{2}-1)^{(m-2)/2}\left(3U_{m}\left(\xi\right)+2U_{m-2}(\xi)+U_{m-4}(\xi)\right)\right|\ge\left|3U_{m}\left(\zeta\right)+2U_{m-2}(\zeta)+U_{m-4}(\zeta)\right|.\label{eq:quartic-Chevinq}
\end{equation}
Let $\zeta=\cos\beta$. Using the fact that $1/3<\zeta^{2}<1/2$ and
$U_{m}(\zeta)=\sin(m+1)\beta/\sin\beta$, we obtain the following
upper bound for the right hand side of \eqref{eq:quartic-Chevinq}:
\begin{eqnarray*}
\left|3U_{m}\left(\zeta\right)+2U_{m-2}(\zeta)+U_{m-4}(\zeta)\right| & \le & 6\sqrt{2}.
\end{eqnarray*}

Since 
\begin{equation}
\alpha=\frac{\pi}{4}\frac{(4k\pm2)}{m+1}\label{eq:alphaform}
\end{equation}
where $k\in\mathbb{Z}$ and $-\pi/4<\alpha<\pi/4$, we have 
\[
|\alpha|\le\frac{\pi}{4}\left(1-\frac{1}{m+1}\right).
\]
Thus
\[
\cos\alpha\ge\frac{\sqrt{2}}{2}\left(\cos\frac{\pi}{4(m+1)}+\sin\frac{\pi}{4(m+1)}\right).
\]
This inequality and \eqref{eq:xiform} give
\begin{equation}
2\cos\theta=4\cos^{2}\alpha-2\ge4\sin\frac{\pi}{2(m+1)}.\label{eq:cosrestriction}
\end{equation}
From the definition of the Chebyshev polynomial, we have 
\begin{eqnarray*}
U_{m-2}(\xi) & = & \frac{\sin(m-1)\alpha}{\sin\alpha}\\
 & = & \frac{\sin(m+1)\alpha\cos2\alpha-\cos(m+1)\alpha\sin2\alpha}{\sin\alpha}\\
 & = & \frac{\sin(m+1)\alpha\cos2\alpha}{\sin\alpha}.
\end{eqnarray*}
With similar computations for $U_{m-4}(\xi)$, we obtain
\[
|3U_{m}\left(\xi\right)+2U_{m-2}(\xi)+U_{m-4}(\xi)|=\frac{|\sin(m+1)\alpha||3+2\cos2\alpha+\cos4\alpha|}{|\sin\alpha|}.
\]
Since $\sin(m+1)\alpha=\pm1$ and $\cos2\alpha\ge0$, the right side
is at least $2\sqrt{2}$. We combine this with \eqref{eq:cosrestriction}
to have 

\begin{eqnarray*}
\left|(2\cos\theta+1)^{(m-2)/2}\left(3U_{m}\left(\xi\right)+2U_{m-2}(\xi)+U_{m-4}(\xi)\right)\right| & \ge & \left(1+4\sin\frac{\pi}{2(m+1)}\right)^{(m-2)/2}2\sqrt{2}\\
 & \ge & 6\sqrt{2}
\end{eqnarray*}
when $m\ge6.$ The inequality \eqref{eq:quartic-Chevinq} follows.
By the Intermediate Value Theorem, we have at least one root when
$\sin(m+1)\alpha$ changes between $-1$ and $1$ with $-\pi/4<\alpha<\pi/4$.
From the formula \eqref{eq:alphaform}, the number of roots $q_{1}$
when $0<\cos\theta<1$ is at least $2(\left\lfloor (m-2)/4\right\rfloor )$.
By symmetry, the number of roots $q_{1}\ne\pm i,1$ with $\Re q_{1}>-1/3$
on the unit circle is at least $4(\left\lfloor (m-2)/4\right\rfloor )$.
Note that each of these roots gives two more roots $q_{2}$ and $q_{3}$
on the quartic curve.

It remains to check the multiplicities of $q_{1}=\pm i,1$ in the
equation \eqref{eq:quartic-q}. We note that this equation has a root
$q_{1}=1$ with multiplicity at least 1. In the case $q_{1}=1$ we
obtain four more roots $q_{2}$, $q_{2}^{-1}$, $q_{3}$, and $q_{3}^{-1}$.
We now consider the case $q_{1}=\pm i$. The equation $q^{2}+(1+q_{1})q+q_{1}^{2}+q_{1}+1=0$
where $q=q_{2},q_{3}$ gives $(q_{2},q_{3})=(-1,i)$ or $(q_{2},q_{3})=(i,-1)$
when $q_{1}=i$ and $(q_{2},q_{3})=(-1,-i)$ or $(q_{2},q_{3})=(-i,-1)$
when $q_{1}=-i$ . Hence each of the roots $q_{1}=\pm i$ gives us
another root at $-1$ with the same multiplicity. To check the multiplicities
at $q_{1}=\pm i$, we need to differentiate the equation \eqref{eq:quartic-q}
with respect to $q_{1}$. We obtain its derivatives by applying implicit
differentiation to the equation $q^{2}+(1+q_{1})q+q_{1}^{2}+q_{1}+1=0$.
After substituting $q_{1}=\pm i$ in \eqref{eq:quartic-q} and its
derivatives, we see that the multiplicity of $\pm i$ is 
\[
\begin{cases}
2 & \mbox{if }m=4k\\
3 & \mbox{if }m=4k+1\\
0 & \mbox{if }m=4k+2\\
1 & \mbox{if }m=4k+3
\end{cases}.
\]

The table below tabulates the $3m-1$ roots of \eqref{eq:quartic-q}.

\begin{center}
\begin{tabular}{|c|c|c|c|c|}
\hline 
 & $m=4k$ & $m=4k+1$ & $m=4k+2$ & $m=4k+3$\tabularnewline
\hline 
\hline 
$q_{1}=e^{i\theta}$,$\Re q_{1}>-1/3$, $q_{1}\ne\pm i,1$ & $3(4k-4)$ & $3(4k-4)$ & $12k$ & $12k$\tabularnewline
\hline 
$q_{1}=1$ & $5$ & $5$ & $5$ & $5$\tabularnewline
\hline 
$q_{1}=\pm i$ & $6$ & $9$ & $0$ & $3$\tabularnewline
\hline 
Total & $12k-1$ & $12k+2$ & $12k+5$ & $12k+8$\tabularnewline
\hline 
\end{tabular}
\par\end{center}

All the roots counted on the table lie on the curves given in the
lemma. The number of roots counted equals the number of possible roots
which is $3m-1$. Also, as a consequence of the Intermediate Value
Theorem applied to the intervals formed by $\sin(m+1)\alpha=\pm1$,
the roots $q_{1}$ are dense on the portion of the unit circle with
real part at least $-1/3$. The lemma follows.

\end{proof}

\begin{theorem}\label{quartic}

Let $H_{m}(z)$ be a sequence of polynomials whose generating function
is
\[
\sum H_{m}(z)t^{m}=\frac{1}{1+B(z)t+A(z)t^{4}}
\]
where $A(z)$ and $B(z)$ are polynomials in $z$ with complex coefficients.
The roots of $H_{m}(z)$ which satisfy $A(z)\ne0$ lie on the curve
$\mathcal{C}_{4}$ defined by
\[
\Im\frac{B^{4}(z)}{A(z)}=0\qquad\mbox{and}\qquad0\le\Re\frac{B^{4}(z)}{A(z)}\le\frac{4^{4}}{3^{3}},
\]
and are dense there as $m\rightarrow\infty$. 

\end{theorem}

\begin{proof}

From the definition of $q$-discriminant in \eqref{eq:qdisc}, we
have 
\[
\mathrm{Disc}_{t}(1+B(z)t+A(z)t^{4};q)=-A^{2}(z)B^{4}(z)q^{3}(1+q+q^{2})^{3}+A^{3}(z)(1+q+q^{2}+q^{3})^{4}.
\]
If $q$ is a quotient of two roots of $1+B(z)t+A(z)t^{4}$, then 
\begin{eqnarray*}
\frac{B^{4}(z)}{A(z)} & = & \frac{(1+q+q^{2}+q^{3})^{4}}{q^{3}(1+q+q^{2})^{3}}.
\end{eqnarray*}
Let $f(q)$ be the function on the right side. We note that $f(q)$
maps $q_{1}=e^{i\theta}$ with $\Re q_{1}\ge-1/3$ to the real interval
$[0,4^{4}/3^{3}]$ since 
\[
f(q_{1})=\frac{(q_{1}^{3/2}+q_{1}^{-3/2}+q_{1}^{1/2}+q_{1}^{-1/2})^{4}}{(q_{1}+q_{1}^{-1}+1)^{3}}.
\]
If $q$ is a point on the quartic curve in Lemma \ref{quarticqlemma}
then $q$ and $q_{1}$ are related by 
\[
q_{1}^{2}+q^{2}+q_{1}q+q_{1}+q+1=0.
\]
Multiplying this equation by $q_{1}-q$, we obtain 
\[
q_{1}^{3}+q_{1}^{2}+q_{1}=q^{3}+q^{2}+q.
\]
Thus by the definition of $f(q)$, we have $f(q)=f(q_{1})$. Since
\[
\mathrm{Disc}_{t}(1+B(z)t+A(z)t^{4})=-3^{3}A^{2}(z)B^{4}(z)+4^{4}A^{3}(z),
\]
the roots of $H_{m}(z)$ lie on the curve $\mathcal{C}_{4}$. The
density of these roots follows from arguments similar to those in
the proof of Theorem \ref{quadratic}. 

\end{proof}

\textbf{Remark: }One may try to find the root distribution of $H_{m}(z)$
in the case $D(t,z)=1+B(z)t+A(z)t^{5}$. From computer experiments,
the distribution of the quotients of roots of $D(t,z)$ in the case
$m=50$ is given in the figure below.

\begin{figure}[H]
\begin{centering}
\includegraphics{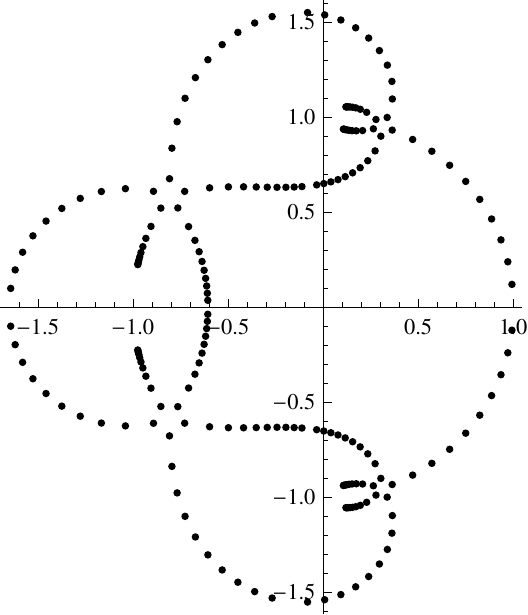}
\par\end{centering}

\caption{Distribution of the quotients of roots of the quintic denominator}
\end{figure}

We end this paper with the following conjecture.

\begin{conjecture}

Let $H_{m}(z)$ be a sequence of polynomials whose generating function
is
\[
\sum H_{m}(z)t^{m}=\frac{1}{1+B(z)t+A(z)t^{n}}
\]
where $A(z)$ and $B(z)$ are polynomials in $z$ with complex coefficients.
The roots of $H_{m}(z)$ which satisfy $A(z)\ne0$ lie on the curve
$\mathcal{C}_{n}$ defined by
\[
\Im\frac{B^{n}(z)}{A(z)}=0\qquad\mbox{and}\qquad0\le(-1)^{n}\Re\frac{B^{n}(z)}{A(z)}\le\frac{n^{n}}{(n-1)^{n-1}},
\]
and are dense there as $m\rightarrow\infty$.

\end{conjecture}


\begin{thebibliography}{10}
\bibitem[1]{aar}G. E. Andrews, R. Askey, and R. Roy, $\textit{Special Functions}$,
Cambridge University Press, Cambridge, 1999.

\bibitem[2]{apostal}T. M. Apostol, The resultants of the cyclotomic
polynomials $F_{m}(ax)$ and $F_{n}(bx)$, Math. Comp. 29 (1975),
1\textendash{}6.

\bibitem[3]{bkw} S. Beraha, J. Kahane, N. J. Weiss, Limits of zeroes
of recursively defined polynomials, Proc. Nat. Acad. Sci. U.S.A. 72
(1975), no. 11, 4209.

\bibitem[4]{bkw-1} S. Beraha, J. Kahane, N. J. Weiss, Limits of zeros
of recursively defined families of polynomials, Studies in foundations
and combinatorics, pp. 213\textendash{}232, Adv. in Math. Suppl. Stud.,
1, Academic Press, New York-London, 1978. 

\bibitem[5]{boyergoh}R. Boyer, W. M. Y. Goh, On the zero attractor
of the Euler polynomials, Adv. in Appl. Math. 38 (2007), no. 1, 97\textendash{}132.

\bibitem[6]{boyergoh-1}R. Boyer, W. M. Y. Goh, Polynomials associated
with partitions: asymptotics and zeros, Special functions and orthogonal
polynomials, 33\textendash{}45, Contemp. Math., 471, Amer. Math. Soc.,
Providence, RI, 2008.

\bibitem[7]{boyergoh-2}R. Boyer, W. M. Y. Goh, Appell polynomials
and their zero attractors, Gems in experimental mathematics, 69\textendash{}96,
Contemp. Math., 517, Amer. Math. Soc., Providence, RI, 2010.

\bibitem[8]{cc}M. Charalambides, G. Csordas, The distribution of
zeros of a class of Jacobi polynomials, Proc. Amer. Math. Soc. 138
(2010), no. 12, 4345\textendash{}4357.

\bibitem[9]{dilcherstolarsky}K. Dilcher and K. B. Stolarsky, Resultants
and discriminants of Chebyshev and related polynomials, Trans. Amer.
Math. Soc. 357 (2005), no. 3, 965\textendash{}981.

\bibitem[10]{hs}M. X. He, E. B. Saff, The zeros of Faber polynomials
for an m-cusped hypocycloid, J. Approx. Theory 78 (1994), no. 3, 410\textendash{}432.

\bibitem[11]{gisheismail}J. Gishe and M. E. H. Ismail, Resultants
of Chebyshev polynomials, Z. Anal. Anwend. 27 (2008), no. 4, 499\textendash{}508.

\bibitem[12]{ismail}M. E. H. Ismail, Difference equations and quantized
discriminants for $q$-orthogonal polynomials, Adv. in Appl. Math.
30 (2003), no. 3, 562\textendash{}589.

\bibitem[13]{tz}K. Tran, A. Zaharescu, Pair correlation of roots
of rational functions with rational generating functions and quadratic
denominators, Ramanujan J. 31 (2013), no. 1, 129\textendash{}145.\end{thebibliography}
\end{document}